\documentclass{amsart}

\usepackage{amstext,amssymb,amsmath,amsbsy,mathrsfs,mathtools,amsfonts,amscd,exscale}
\usepackage{graphics, epstopdf}
\usepackage{color}
\usepackage{amsmath, amssymb}
\usepackage{enumerate}
\usepackage{bigints}
\usepackage{verbatim}
\usepackage[poorman]{cleveref}

\usepackage{graphicx}
\usepackage{epstopdf}
\usepackage{subcaption}
\usepackage{float}
\usepackage{placeins}

\usepackage{bm}

\newcommand{\ds}{\displaystyle}

\newcommand{\ve}{\varepsilon}

\renewcommand{\sc}{\textsc}

\newcommand{\MA}{Monge-Amp\`{e}re }
\newcommand{\Th}{\mathcal{T}_h}

\newcommand{\cC}{\mathcal{C}}
\newcommand{\cS}{\mathcal{S}}

\newtheorem{Corollary}{Corollary}[section]
\newtheorem{Theorem}{Theorem}[section]
\newtheorem{Lemma}[Theorem]{Lemma}
\newtheorem{Proposition}[Theorem]{Proposition}

\theoremstyle{definition}
\newtheorem{Definition}[Theorem]{Definition}

\newtheorem{remark}[Theorem]{Remark}

\numberwithin{equation}{section}

\def\l({\left(}
\def\r){\right)}
\def\uve{u_{\varepsilon}}

\def\ds{\partial}

\def \tr{{\rm tr}}

\def\Nhi{\mathcal{N}_h^0}
\def\Nhb{\mathcal{N}_h^b}
\def\sdd{\nabla^2_{\delta}} 
\def\sddp{\nabla^{2,+}_{\delta}} 
\def\sddm{\nabla^{2,-}_{\delta}} 

\def\Vperp{\mathbb{S}^{\perp}}
\def\Vperpt{\mathbb{S}^{\perp}_{\theta}}
\def\St{\mathbb S_{\theta}}
\def\Vh{\mathbb{V}_h}

\def\Cka{C^{2+k,\alpha}}
\def\Ctao{C^{2,\alpha}(\overline{\Omega})}
\def\Cthao{C^{3,\alpha}(\overline{\Omega})}
\def\Wti{W^2_{\infty}}
\def\Wtio{W^2_{\infty}(\Omega)}
\def\Wspo{W_p^s(\Omega)}
\def\Bhi{B_i}
\def\interp{\mathcal I_h}

\def\dep{d_{\varepsilon}}

\definecolor{red}{rgb}{1,0,0}
\definecolor{blue}{rgb}{0,0,1}

\begin{document}
	
	\title[Two-scale method for the Monge-Amp\`ere Equation]{Two-scale
		method for the Monge-Amp\`ere Equation: Pointwise Error Estimates}
	
	\author{R. H. Nochetto}
	\address{Department of Mathematics, University of Maryland, College Park, Maryland 20742}
	\email{rhn@math.umd.edu}
	\thanks{The first author was partially supported by the NSF Grant DMS -1411808, the
		Institut Henri Poincar\'e (Paris), and the Hausdorff
                Institute (Bonn).}
	
	\author{D. Ntogkas}
	\address{Department of Mathematics, University of Maryland, College Park, Maryland 20742}
	\email{dimnt@math.umd.edu}
	\thanks{The second author was partially supported by the  NSF Grant DMS -1411808 and
		the 2016-2017 Patrick and Marguerite Sung Fellowship of the
		University of Maryland.}
	
	\author{W. Zhang}
	\address{Department of Mathematics, Rutgers University of Maryland, New Brunswick, New Jersey 08854}
	\email{wujun@math.rutgers.edu}
	\thanks{The third author was partially supported by the  NSF Grant DMS -1411808 and the
		Brin Postdoctoral Fellowship of the University of Maryland.}

	\begin{abstract}
	 In this paper we continue the analysis of the two-scale method for the \MA 
	 equation for dimension $d \geq 2$ introduced in \cite{NoNtZh}. 
	 We prove continuous dependence of discrete solutions on data
         that in turn hinges on a discrete version of the Alexandroff estimate. 
	 They are both instrumental to prove pointwise error estimates for 
	 classical solutions with H\"older and Sobolev regularity.
	 We also derive convergence rates for viscosity solutions
         with bounded Hessians which may be piecewise smooth or degenerate.
	\end{abstract}
	
	\maketitle

	
	\vspace{0.2cm}
	

	\section{Introduction}
	We consider the Monge-Amp\`ere equation with Dirichlet boundary condition
	\begin{equation} \label{E:MA}
	\left\{
	\begin{aligned}
	\det{D^2u}&=f  & {\rm in} \ &\Omega \subset \mathbb {R}^d,
	\\ u &=g & {\rm on} \ &\partial \Omega,
	\end{aligned}
	\right.
	\end{equation}
	where $f \geq 0$ is uniformly continuous, $\Omega$ is a uniformly
	convex domain and $g$ is a continuous function. We seek a
	\textit{convex} solution $u$ of \eqref{E:MA}, which is critical for \eqref{E:MA} to be elliptic and 
	have a unique viscosity solution \cite{Gut}. 
	
	The Monge-Amp\`ere equation has a wide spectrum of applications, which has led to an increasing interest 
	in the investigation of efficient numerical methods. There are several existing methods for 
	the Monge-Amp\`ere equation, as described in \cite{NoNtZh}. Error estimates in $H^1(\Omega)$ are
	established in \cite{BrenNeil, BrenNeil2} for solutions with $H^3(\Omega)$ regularity or more. 
	Awanou \cite{Aw} also proved a linear rate of convergence for classical
	solutions for the wide-stencil method, when applied to a perturbed
	Monge-Amp\`ere equation with an extra lower order term $\delta u$;
	the parameter $\delta > 0$ is independent of the mesh and appears in
	reciprocal form in the rate.
	
	On the other hand, Nochetto and Zhang followed an
          approach based on the discrete Alexandroff estimate developed 
	in \cite{NoZh} and established pointwise error estimates in
        \cite{NoZh2} for the method of Oliker and Prussner \cite{OlPr}.
        In this paper we follow a similar approach and derive 
	pointwise rates of convergence for classical solutions of \eqref{E:MA} that have 
	H\"older or Sobolev regularity and for viscosity
        solutions with bounded Hessians which may be piecewise smooth or
        degenerate.
        
It is worth mentioning a rather strong connection between the 
semi-Lagrangian method of Feng and Jensen \cite{FeJe:16} and our 
two-scale approach introduced in \cite{NoNtZh}. In fact, for an appropriate choice of discretization of symmetric positive semidefinite matrices with trace one, discussed in \cite{FeJe:16} along with the implementation, one can show that the discrete solutions of both methods coincide. Therefore, the error estimates in this paper extend to the fully discrete method of \cite{FeJe:16}. This rather surprising equivalence property is fully derived in a forthcoming paper, along with optimal error estimates in special cases via enhanced techniques for pointwise error analysis.

	\subsection{Our contribution}
	
	The two-scale method was introduced in \cite{NoNtZh} and hinges on the following
	formula for the determinant of the semi-positive Hessian $D^2w$ of a smooth
	function $w$, first suggested by Froese and Oberman \cite{FrOb1}:
	\begin{equation}\label{E:Det}
	\det{D^2w}(x) = \min\limits_{(v_1,\ldots,v_d) \in \Vperp} \prod_{j=1}^d v_j^TD^2w(x) \ v_j ,
	\end{equation}
	where $\Vperp$ is the set of all $d-$orthonormal bases in $\mathbb{R}^d$.
	To discretize this expression, we impose our discrete
        solutions to lie on a space of 
	continuous piecewise linear functions over an unstructured quasi-uniform
        mesh $\Th$ of size $h$; this defines the fine scale. The
        mesh also defines the computational domain $\Omega_h$, which
        we describe in more detail in \Cref{S:KeyProperties}. The
        coarser scale $\delta$ corresponds to the length of
        directions used to approximate
	 the directional derivatives that appear in  \eqref{E:Det}, namely
	 \[
	 \sdd  w (x;v) := \frac{ w(x+\delta v) -2w(x) +w(x-\delta v) }{ \delta^2}
	 \quad
	 \text{and}
	 \quad
	 |v| = 1 ,
	 \] 
	 for any $w\in C^0(\overline{\Omega})$;
	 To render the method practical, we introduce a discretization
         $\Vperpt$ of the set $\Vperp$ governed by the
	 parameter $\theta$ and denote our discrete solution by $\uve$, 
	 where $\varepsilon = (h,\delta, \theta)$ represents the scales of the method 
	 and the parameter $\theta$. We define the discrete \MA operator to be
	 $$
	 	T_{\varepsilon}[\uve](x_i):=\min\limits_{\mathbf{v} \in \Vperpt} \left( \prod_{j=1}^d \sddp \uve(x_i;v_j) - \sum_{j=1}^d \sddm \uve (x_i;v_j) \right) ,
	 $$
         where $\nabla_\delta^{2,\pm}$ are the positive and negative
         parts of $\nabla_\delta^2$.
	 In \Cref{S:KeyProperties} we review briefly the role of each term in 
	 the operator $T_{\varepsilon}$ and recall some key properties
         of $T_{\varepsilon}$.

	The merit of this definition of $T_{\varepsilon}$ is that it
	leads to a clear separation of scales, which is a key theoretical advantage
	over the original wide stencil method of \cite{FrOb1}. This also
	yields continuous dependence of discrete solutions on data, namely
        \Cref{P:ContDep}, which allows us to prove
	rates of convergence in $L^{\infty}(\Omega)$ for our method
        depending on the regularity of $u$;
        this is not clear for the wide stencil method of \cite{FrOb1}.
        Moreover, the two-scale method is formulated over
        unstructured meshes $\Th$, which adds flexibility to
        partition arbitrary uniformly convex domains $\Omega$.
        This is achieved at the expense of points
        $x_i\pm\delta v_j$ no longer being nodes of $\Th$, which is
        responsible for an additional interpolation error in the
        consistency estimate of $T_\varepsilon$. To locate such points and
        evaluate $\sdd \uve(x_i;v_j)$, we resort to fast search techniques
        within \cite{WalkerPaper,WalkerWeb} and thus render the
        two-scale method practical.
	Compared with the error analysis of the Oliker-Prussner method
        \cite{NoZh}, we do not require $\Th$ to be cartesian.
        	
	In \cite{NoNtZh} we prove existence and uniqueness of a discrete solution
	for our method, and convergence to the viscosity solution of
        \eqref{E:MA}, without regularity beyond uniform continuity of $f$ and $g$. This entails dealing with the $L^\infty-$norm and using the discrete comparison principle for piecewise linear functions (monotonicity). Within this $L^\infty$ framework and under the regularity requirement $u \in W^2_\infty(\Omega)$, we now prove rates of convergence for classical solutions with either
	H\"older or Sobolev regularity and for a special class of viscosity solutions. Therefore, our two-scale method \cite{NoNtZh} and the Oliker-Prussner method \cite{OlPr,NoZh2} are the only schemes known to us to converge to the viscosity solution and have provable rates of convergence.

	The first important tool for proving pointwise rates of convergence is the
	discrete Alexandroff estimate introduced in \cite{NoZh}:
	if $w_h$ is an arbitrary continuous piecewise linear function,
        $w_h\ge0$ on $\partial\Omega_h$, and
	$\Gamma w_h$ stands for its convex envelope, then
	\begin{equation*}
	\max\limits_{ x_i \in \mathcal{N}_h } w_h^-(x_i) \leq C \left( \sum\limits_{x_i \in C_{-}(w_h)} \left| \ds \Gamma w_h (x_i) \right| \right)^{1/d}
	\end{equation*}
	where $\partial \Gamma w_h$ is the subdifferential of $\Gamma w_h$ and
	$C_{-}(w_h)$ represents the lower contact set of $w_h$, i.e. the set
	of interior nodes $x_i\in\Nhi$ such that $\Gamma w_h(x_i)=w_h(x_i)$;
	hereafter we write $w_h^-(x_i) := - \min\{w_h(x_i),0\}$.
	To control the measure of the subdifferential at each node, we show the following estimate
	$$
	|\ds w_h(x_i)| \leq \delta^d \min\limits_{(v_1,\ldots,v_d)\in \Vperp}
	\prod_{j=1}^d  \sdd w_h(x_i;v_j)
        \quad\forall \, x_i\in\Nhi,
	$$
	such that the ball centered
	at $x_i$ and of radius $\delta$ is contained in $\Omega_h$.
	Combining both estimates, we derive the following continuous dependence estimate
	$$
	\max_{\Omega_h} \, (u_h-w_h)^- \leq C \delta
	\left( \sum\limits_{x_i \in C_{-}(u_h-w_h)} \left( T_\ve[u_h](x_i)^{1/d}-T_\ve[w_h](x_i)^{1/d}  \right)^d \right)^{1/d}  
	$$
	for all continuous piecewise linear functions $u_h$ and $w_h$ such
	that $T_\ve[u_h](x_i)\ge0$ and $T_\ve[w_h](x_i)\ge 0$ for all $x_i\in\Nhi$.
	This result is instrumental and, combined with operator consistency and a
	discrete barrier argument close to the boundary, eventually leads to
	the following pointwise error estimates
	$$
	\|\uve-u\|_{L^\infty(\Omega_h)} \leq C(d, \Omega,f,u) \ h^{\frac{\alpha+k}{\alpha+k+2}}
	$$
	provided $u\in C^{2+k,\alpha}(\overline{\Omega})$ with $0<\alpha\le1$
	and $k=0,1$, as well as
	\[
	\|\uve-u\|_{L^\infty(\Omega_h)} \leq C(d, \Omega,f,u) \ h^{1-\frac{2}{s}}
	\]
	provided $u\in W^s_p(\Omega)$ with $2+d/p<s\le4$ and $p > d$, and
	$\delta$ is suitably chosen in terms of $h$; see Theorems
	\ref{T:RatesHolder} and \ref{T:RatesSobolev}.
	We also consider a special case of viscosity solutions with
        bounded but
	discontinuous Hessians, and manage to prove a rate of
	convergence (see Theorem \ref{T:RatesPW}). Since these theorems 
	are proven under the nondegeneracy assumption $f > 0$,
        we examine in \Cref{T:RatesDegen} the effect of degeneracy $f \geq 0$.
	In \cite{NoNtZh} we explore numerically both classical and $W^2_\infty$ viscosity
	solutions and observe linear rates with respect to $h$ for both cases, which are better
	than predicted by this theory.

	\subsection{Outline}
	We start by briefly presenting the operator
        $T_\varepsilon$ in \Cref{S:KeyProperties} 
	and recalling some important results from \cite{NoNtZh}. 
	In \Cref{S:dAle} we mention the discrete Alexandroff estimate and
	combine it in \Cref{S:CoDeDa} with some geometric estimates to
	obtain the continuous dependence of the discrete solution on
	data. This is much stronger than stability,
	and is critical to prove rates of convergence for fully nonlinear
	PDEs. Lastly, in \Cref{S:RoC}  we combine this result with
        operator consistency and a discrete
	barrier argument close to the boundary to
	derive rates of convergence upon making
	judicious choices of $\delta$ and $\theta$ in terms of $h$.

	
	\section{Key Properties of the Discrete Operator}\label{S:KeyProperties}
	We recall briefly some of the key properties of operator $T_\varepsilon$,
        as proven in \cite{NoNtZh}.

	\subsection{Definition of $T_{\varepsilon}$} \label{S:MonotoneDefinition}
	Let $\mathcal{T}_h$ be a shape-regular and quasi-uniform triangulation
	with meshsize $h$. The computational domain $\Omega_h$
        is the union of elements of $\Th$ and $\Omega_h\ne\Omega$.
	If $\mathcal{N}_h$ denotes the nodes of $\mathcal{T}_h$, then
	$\mathcal{N}_h^b := \{x_i \in \mathcal{N}_h: x_i \in \partial \Omega_h\}$
	are the boundary nodes and
	$\mathcal{N}_h^0 := \mathcal{N}_h \setminus \mathcal{N}_h^b$
	are the interior nodes. We require that $\Nhb \subset \partial \Omega$, which in view of the convexity of $\Omega$ implies that $\Omega_h$ is also convex and $\Omega_h \subset \Omega$.
	We denote by $\Vh$ the space of continuous
	piecewise linear functions over $\mathcal{T}_h$.
	We let $\Vperp$ be the collection of all
	$d$-tuples of orthonormal bases and
	$\mathbf{v} := (v_1,\ldots,v_d) \in \Vperp$ be a generic element, whence 
	each component $v_i\in\mathbb{S}$, the unit sphere $\mathbb{S}$ of
        $\mathbb{R}^d$. We next introduce a finite subset
        $\mathbb{S}_\theta$ of $\mathbb{S}$ governed by the angular parameter $\theta>0$:
        given $v \in \mathbb S$, there
	exists $v^{\theta}\in\mathbb{S}_\theta$ such that
	$$
	|v-v^{\theta}| \leq \theta.
	$$
	Likewise, we let $\Vperpt\subset\Vperp$
        be a finite approximation of $\Vperp$: for any
	$\mathbf v = (v_j)_{j=1}^d \in \Vperp$ there exists
	$\mathbf v^{\theta} = (v_j^{\theta})_{j=1}^d \in \Vperpt$
	such that $v_j^\theta \in \St$ and $|v_j - v_j^{\theta}| \leq \theta $ for all $1 \leq j \leq d$
	and conversely.
	
	For $x_i\in\mathcal{N}_h^0$, we use centered second
        differences with a coarse scale $\delta$
	\begin{equation} \label{E:2Sc2Dif}
        \sdd w(x_i;v_j) :=   \frac{ w(x_i+ \hat\delta v_j) -2 w(x_i) +
          w(x_i- \hat\delta v_j) }{ \hat\delta^2}
	\end{equation}
	where $\hat\delta:=\rho\delta$ with $0 < \rho \le 1$ the
        biggest number such that the ball centered at $x_i$ of radius
        $\hat\delta$ is contained in $\Omega_h$; we stress that
        $\rho$ need not be computed exactly.
	This is well defined for any
	$w \in C^0(\overline{\Omega})$, in particular for $w\in\Vh$.
	We define $\varepsilon:= (h,\delta, \theta)$ and we seek $\uve \in \Vh$ such that $u^{\varepsilon}(x_i)=g(x_i)$ for
	$x_i \in \mathcal{N}_h^b$ and for $x_i \in \mathcal{N}_h^0$
	\begin{equation} \label{E:2ScOp}
	T_{\varepsilon}[\uve](x_i):=\min\limits_{\mathbf{v} \in \Vperpt} \left( \prod_{j=1}^d \sddp \uve(x_i;v_j) - \sum_{j=1}^d \sddm \uve (x_i;v_j) \right) = f(x_i),
	\end{equation}
	where we use the notation
          $$
		\sddp \uve(x_i;v_j) = \max{(\sdd \uve(x_i;v_j),0)},
                \quad \sddm \uve(x_i;v_j) = -\min{(\sdd \uve(x_i;v_j),0)}
          $$
	to indicate positive and negative parts of
	the centered second differences.

        
        \subsection{Key Properties of $T_\varepsilon$} \label{S:PropertiesMonotone}
	One of the critical properties of the \MA equation is the convexity of the solution $u$. 
	The following notion mimics this at the discrete level.
        
	\begin{Definition}[discrete convexity]\label{D:discrete-convexity}
		We say that $w_h \in \mathbb V_h$ is discretely convex  if
		$$
		\sdd w_h(x_i;v_j) \geq 0 \qquad \forall x_i \in \Nhi, \quad
		\forall v_j \in \St.
		$$
	\end{Definition}
	The following lemma guarantees the discrete convexity of 
	subsolutions of \eqref{E:2ScOp} \cite[Lemma 2.2]{NoNtZh}.
	\begin{Lemma}[discrete convexity]\label{L:DisConv}	
		If  $w_h \in \Vh$ satisfies
		\begin{equation} \label{E:Oper}
		T_{\varepsilon}[w_h](x_i) \geq 0 \quad \forall x_i \in \mathcal{N}_h^0,
		\end{equation}
		then $w_h$ is \textit{discretely convex} and as a consequence
		\begin{equation}\label{E:simpler-def}
		T_{\varepsilon} [w_h](x_i)= \min_{\mathbf{v} \in \Vperpt} \prod_{j=1}^d \sdd w_h(x_i;v_j),
		\end{equation}
		namely 
		$$
		\sddp w_h(x_i;v_j) = \sdd w_h(x_i;v_j),
		\quad
		\sddm w_h(x_i;v_j) =0
		\quad\forall x_i\in \mathcal{N}_h^0,\quad\forall v_j\in \St.
		$$
		Conversely, if $w_h$ is discretely convex, then (\ref{E:Oper}) is valid.
	\end{Lemma}
	Another important property of operator $T_\varepsilon$
        that relies on its monotonicity is the following 
	discrete comparison principle \cite[Lemma 2.4]{NoNtZh}.
	\begin{Lemma}[discrete comparison principle] \label{L:DCP}
	  Let $u_h,w_h \in \Vh$ with $u_h \leq w_h$ on the
        discrete boundary $\partial \Omega_h$ be such that 
		\begin{equation}\label{E:comparison}
		T_{\varepsilon}[u_h](x_i) \geq T_{\varepsilon}[w_h](x_i)  \geq 0 \ \ \ \forall x_i \in \Nhi. 
		\end{equation}
		Then, $u_h \leq w_h$ everywhere.
	\end{Lemma}
	We now state a consistency estimate, proved in \cite[Lemma 4.1]{NoNtZh},
        that leads to
	pointwise rates of convergence. To this end, given a node $x_i \in \Nhi$, we denote by
	\begin{equation}\label{E:Bi}
	B_i := \cup \{\overline{T}: T\in\Th, \, \textrm{dist }(x_i,T) \le \hat\delta\}
	\end{equation}
	where $\hat\delta$ is defined in \eqref{E:2Sc2Dif}.
	We also define the $\delta$-interior region
	\begin{equation}\label{Omega-delta}
	\Omega_{h,\delta} = \left\{ T \in \mathcal{T}_h \ : \  {\rm dist}
	(x, \partial \Omega_h ) \geq \delta  \ \forall x \in T  \right\},
	\end{equation}
	and the $\delta$-boundary region:
	$$
	\omega_{h,\delta} = \Omega \setminus  \Omega_{h,\delta}.
	$$
	\begin{Lemma}[{consistency of $T_{\varepsilon} [\mathcal{I}_h u]$}] \label{L:FullConsistency}
		Let $x_i \in \Nhi \cap \Omega_{h,\delta}$ and $B_i$ be defined as in
		\eqref{E:Bi}. If $u \in C^{2+k,\alpha}(B_i)$ with $0<\alpha\leq 1$ and
		$k=0,1$ is convex, and $\mathcal I_h u $ is its piecewise linear interpolant, then 
		\begin{equation}\label{E:FullConsistency}
		\left|  \det D^2u(x_i) - T_{\varepsilon}[\mathcal I_h u] (x_i)  \right| \leq C_1(d,\Omega,u) \delta^{k+\alpha} + C_2(d,\Omega,u) \left( \frac{h^2}{\delta^2} + \theta^2 \right),
		\end{equation}
		where
		\begin{equation}\label{E:C1-C2}
		C_1(d,\Omega,u)= C |u|_{C^{2+k,\alpha}(B_i)} |u|_{W^2_{\infty}(B_i)}^{d-1}, \quad C_2(d,\Omega,u) = C |u|_{W^2_{\infty}(B_i)}^d.
		\end{equation}
		If $x_i\in\Nhi$ and $u \in W^2_{\infty}(B_i)$, then (\ref{E:FullConsistency}) remains valid with $\alpha=k=0$ and $\Cka(\Bhi)$ replaced by $\Wti(\Bhi)$.
	\end{Lemma}
	
	\section{Discrete Alexandroff Estimate} \label{S:dAle}
	
	In this section, we review several concepts related to convexity
	as well as the discrete Alexandroff estimate of \cite{NoZh}.
    We first recall several definitions.

	\begin{Definition}[subdifferential] \label{D:ConvEnv} \ 
		\begin{enumerate}[$(i)$]
			\item  The subdifferential of a function $w$
                          at a point
                          $x_0\in \Omega_h$ is the set
			\[ \partial w(x_0) := \left\{ p \in \mathbb{R}^d: \ w(x) \geq w(x_0) + p \cdot (x-x_0),  \  \   \forall x \in \Omega_h \right\}. \]
			\item The subdifferential of a function $w$ on
                          set $E \subset
                          \Omega_h$ is $\partial u(E) :=  \cup_{x \in E}  \partial w(x)$. 
		\end{enumerate}
	\end{Definition}
	\begin{Definition}[convex envelope and discrete lower contact set]\label{def:CEandLCS} \
		\begin{enumerate}[$(i)$]
			\item The convex envelope $\Gamma u$ of a function $w$ is defined to be
			\[
			\Gamma w (x) := \sup_{L } \{ L(x), \; L(y) \leq w(y) \text{ for all $y \in \Omega_h$ and $L$ is affine} \}.
			\]
			\item The discrete lower contact set $C_-(w_h)$ of a function $w_h\in\Vh$ is the set of nodes where the function coincides with its convex envelope, i.e.
			\[
			\cC_- (w_h) := \big\{x_i \in \Nhi: \Gamma w_h(x_i) = w_h(x_i) \big\}.
			\]
			
		\end{enumerate}
	\end{Definition}
	
	\begin{remark}[$w_h$ dominates $\Gamma w_h$] \label{R:2DifConvEnv}
		Since $w_h\ge\Gamma w_h$, at a contact node $x_i\in\cC_-(w_h)$ we have
		$$
		\sdd  \Gamma w_h(x_i;v_j) \leq \sdd
                w_h(x_i;v_j)(x_i)\qquad \forall v_j \in \St.
		$$
	\end{remark}
	
	\begin{remark}[minima of $w_h$ and $\Gamma w_h$] \label{R:MinConvEnv}
		A consequence of Definition
		\ref{def:CEandLCS} (convex envelope and discrete lower contact set)
		is that the minima of $w_h\in\Vh$ and $\Gamma w_h$ are attained at the same
		contact nodes and are equal. 
	\end{remark}
	
	We can now present the discrete Alexandroff estimate from
          \cite{NoZh},
          which states that the minimum of a discrete function is controlled by
          the measure of the subdifferential of its convex envelope in the
          discrete contact set.
	
	\begin{Proposition} [discrete Alexandroff estimate  \cite{NoZh}] \label{P:DAE} 
		Let $v_h$ be a continuous piecewise linear function that satisfies $v_h \geq 0$ on $\partial \Omega_h$. Then,
		\begin{equation*}
		\max \limits_{ x_i \in \Nhi} v_h(x_i)^- \leq C \left( \sum\limits_{x_i \in \cC_{-}(v_h)} \left| \ds \Gamma v_h (x_i) \right| \right)^{1/d}
		\end{equation*}
		where $C=C(d,\Omega)$ depends only on the dimension $d$ and the domain $\Omega$.
		
	\end{Proposition}

	\section{Continuous Dependence on Data} \label{S:CoDeDa}
	
		We derive the continuous dependence of the discrete
		solution on data in Section \ref{S:cont-depend}, which is essential to prove rates of
		convergence. To this end, we first prove a stability estimate in the
		max norm in Section \ref{S:stab} and the concavity of the discrete
		operator in Section \ref{S:concavity}.

	\subsection{Stability of the Two-Scale Method.}\label{S:stab}
	%
		We start with some geometric estimates. The
		first and second lemmas connect the discrete
		Alexandroff estimate with the 2-scale method. They allow us to estimate the
		measure of the subdifferential of a discrete function $w_h$ in terms of our discrete
		operator $T_\varepsilon[w_h]$, defined in \eqref{E:2ScOp}. 
	
	\begin{Lemma} [subdifferential vs hyper-rectangle] \label{L:SubDifBound}
		Let $w \in C^0(\overline{\Omega}_h)$ be convex
                and $x_i\in\Nhi$ be so
		that $x_i\pm\hat\delta v \in \overline{\Omega}_h$ for all
		$v\in \St$ with $\hat\delta\le\delta$. If
		$\mathbf{v}=(v_j)_{j=1}^d\in\Vperpt$ and
		\[
		\alpha_{i,j}^\pm := \frac{w(x_i\pm \hat\delta v_j) - w(x_i)}{\hat\delta}
		\quad\forall \, 1\le j\le d,
		\]
		then
		\[
		\ds w(x_i) \subset  \left\{  p \in \mathbb{R}^d: \
		\alpha_{i,j}^- \le p\cdot v_j \le \alpha_{i,j}^+  \ 1\le j \le d \right\}.
		\]
	\end{Lemma}
	\begin{proof}
		Take $p \in \ds  w(x_i)$ and write
		$$
		w(x) \geq w(x_i) + p \cdot (x-x_i) \ \ \forall x \in \overline{\Omega}_h.
		$$
		Consequently, for any $1\le j \le d$ we infer that
		\[
		w(x_i + \hat\delta v_j) \geq w(x_i) + \hat\delta \ p \cdot v_j,
		\quad
		w(x_i - \hat\delta v_j) \geq w(x_i) - \hat\delta \ p \cdot v_j,
		\]
		or equivalently
		\[
		\frac{w(x_i)-w(x_i-\hat\delta v_j)}{\hat\delta } \leq p \cdot v_j
		\leq \frac{w(x_i+\hat\delta v_j)-w(x_i)}{\hat\delta}.
		\]
		This implies that $p$ belongs to the desired set.
	\end{proof}
	
	\begin{Lemma}[hyper-rectangle volume] \label{L:HypRectVol}
		For d-tuple $\mathbf{v} = (v_j)_{j=1}^d \in \Vperpt$ the volume of the set 
		$$
		K= \left\{ p \in \mathbb{R}^d: \ a_j \leq p \cdot v_j \leq b_j, \ \ j=1,\ldots,d \right\} 
		$$
		is given by 
		$$
		|K|=\prod_{j=1}^d (b_i-a_i).
		$$
	\end{Lemma}
	\begin{proof}
		Let $V=[v_1, \cdots,v_d]\in\mathbb{R}^{d\times d}$ be the orthogonal
		matrix whose columns are the elements of $\mathbf{v}$; hence
		$v_j=V e_j$ where $\left\{ e_j \right\}_{j=1}^d $ is the canonical basis in $\mathbb{R}^d$.
		We now seek a more convenient representation of $K$
		\begin{align*}
		K &=  \left\{ p \in \mathbb{R}^d: \ a_j \leq p \cdot (Ve_j) \leq b_j, \ \ j=1,\ldots,d \right\} 
		\\ &= V^{-T} \left\{ x \in \mathbb{R}^d: \ a_j \leq x \cdot e_j \leq b_j, \ \ j=1,\ldots,d \right\} 
		= V^{-T} \widetilde{K},
		\end{align*}
		whence
		\[
		|K| = |\det{V^{-T}} | \ |\widetilde{K}|=|\widetilde{K}|
		= \prod_{j=1}^d (b_j-a_j),
		\]
		because $\widetilde{K}$ is an orthogonal hyper-rectangle.
	\end{proof}
	
%
	Combining Lemmas \ref{L:SubDifBound} and \ref{L:HypRectVol}
        we get the following corollary. 
	\begin{Corollary}[subdifferential vs discrete operator] \label{C:SubDifOper}
		For every $x_i \in \Nhi \cap \Omega_{h,\delta}$ and a convex function $w$ we have that 
		\[ | \partial w(x_i)| \leq \left( 
		\min\limits_{\mathbf{v} \in \Vperpt} \prod_{j=1}^d  \sdd w(x_i;v_j)
		\right) \delta^d .
                \]
	\end{Corollary} 

	\begin{remark} [artificial factor $\frac{\delta}{h}$]\label{R:VolumeSubopt}
	The above estimate is critical in deriving \Cref{P:ContDep} (continuous dependence on data) and subsequently rates of convergence in \Cref{S:RoC}. We thus wish to provide here intuition about the reduced rates of convergence of \Cref{T:RatesHolder} (rates of convergence for classical solutions) relative to the numerical experiments in \cite{NoNtZh}. To this end, we let $w \in C^2(\overline\Omega)$ be a convex function. Then, \Cref{C:SubDifOper} implies that 
	$$
	T_\varepsilon[\interp w] (x_i) \geq \frac{1}{\delta^d}  | \partial \interp w(x_i)|.
	$$
	However, it was shown by Nochetto and Zhang in \cite[Proposition 5.4]{NoZh2} that
	$$
	| \partial \interp w(x_i)| \geq C h^d \det{(D^2w(x_i))},
	$$
	provided the mesh $\mathcal{T}_h$ is translation invariant, whence
	$$
\det{(D^2w(x_i))} \leq C  \frac{\delta^d}{h^d} 	T_\varepsilon[\interp w] (x_i).
	$$	
	We can now see that using this estimate introduces an extra factor $\frac{\delta}{h} \gg 1$, which could possibly explain the suboptimal rate proved in \Cref{T:RatesHolder} (rates of convergence for classical solutions).
	\end{remark}

	\begin{Lemma} [stability] \label{L:Stability}
		If $w_h\in\Vh$ is $w_h \geq 0$ on $\partial\Omega_h$, then
		$$
		\max_{x_i \in \Nhi} w_h(x_i)^- \leq C \delta \left( \sum_{x_i \in \cC_-(w_h)} 
		T_{\varepsilon}[w_h](x_i) \right)^{1/d} .
		$$
	\end{Lemma} 
	\begin{proof}
		Since the function $w_h\ge0$ on $\partial\Omega_h$, 
		we invoke Proposition \ref{P:DAE} (discrete Alexandroff estimate) for $w_h$ to obtain
		\begin{equation*} \label{E:dAe}
		\max \limits_{ x_i \in \Nhi} w_h(x_i)^- \leq C \left( \sum\limits_{x_i \in \cC_-(w_h)} \left| \ds \Gamma w_h (x_i) \right| \right)^{1/d}
		\end{equation*}
		Applying \Cref{C:SubDifOper} (subdifferential vs discrete operator) to
		the convex function $\Gamma w_h (x_i)$ at a contact point
		$x_i \in \cC_-(w_h)$ and recalling \Cref{R:2DifConvEnv}, we have
		\begin{align*}
		\left| \ds \Gamma w_h (x_i) \right|  &\leq  \delta^d \min\limits_{\mathbf{v} \in \Vperpt} \prod_{j=1}^d  \sdd \Gamma w_h(x_i;v_j)
		\leq \delta^d \min\limits_{\mathbf{v} \in \Vperpt} \prod_{j=1}^d  \sdd  w_h(x_i;v_j)
		= \delta^d T_{\varepsilon}[w_h](x_i),
		\end{align*}
		where the last equality follows from \Cref{L:DisConv} (discrete convexity).
	\end{proof}
	
	\subsection{Concavity of the Discrete Operator}\label{S:concavity}
	%
	
		We recall concavity properties of $(\det A)^{1/d}$ for symmetric positive semi-definite
		matrices $A$ and extend them to $T_{\varepsilon}$. The results can be traced back to \cite{Krylov,Lions}, but we present them here for completeness.

	\begin{Lemma} [concavity of determinant] \label{L:Concavity} The following two statements are valid.
		\begin{enumerate}[$(i)$]
			\item For every symmetric positive semi-definite (SPSD) matrix $A$ we have that
			\[ (\det{A})^{1/d} = \frac{1}{d} \inf{\left\{ \tr(AB) \ \Big| \ B  \  { \rm is \ SPD \ and } \ \det{B}=1 \right\} } \]
			\item The function $A \mapsto (\det{A})^{1/d}$ is concave on SPSD matrices. 
		\end{enumerate}
	\end{Lemma}
	\begin{proof}
		
		We proceed in three steps.
		
		\medskip
		{\it Step 1: Proof of (i) for $A$ invertible.}
		Let $B$ be SPD with $\det B = 1$. Then $B^{1/2}$ is well defined,
		$\det(B^{1/2})=1$ and we obtain
		\[
		\det A = \det(B^{1/2} A B^{1/2}).
		\]
		Let $P$ be an orthogonal matrix that converts $B^{1/2} A B^{1/2}$ into
		a diagonal matrix $D$, namely $D = PB^{1/2} A B^{1/2}P^T$. Applying
		the geometric mean inequality yields
		\[
		\det(B^{1/2} A B^{1/2})^{1/d} = (\det D)^{1/d} \le \frac{1}{d} \tr D
		= \frac{1}{d} \tr(B^{1/2} A B^{1/2}) = \frac{1}{d} \tr(AB),
		\]
		where we have used the invariance of the trace under cyclic
		permutations of the factor to write the last two equalities. This
		shows that
		\[
		(\det{A})^{1/d} \leq \frac{1}{d} \inf{\left\{ \tr(AB) \ \Big| \  B \  { \rm is \ SPD \ and } \ \det{B}=1 \right\} }
		\]
		This inequality is actually equality provided $A$ is invertible. In
		fact, we can take $B = (\det A)^{1/d} A^{-1}$, which is SPD and $\det B = 1$.
		This proves (i) for $A$ nonsingular.
		
		\medskip
		{\it Step 2: Proof of (i) for $A$ singular}.
                Given the singular value decomposition of $A$
		\[
		A = \sum_{i=1}^d \lambda_i v_i \otimes v_i,
		\quad \lambda_1 \ge \cdots \lambda_k > \lambda_{k+1} = \cdots =
		\lambda_d = 0,
		\]
		with orthogonal vectors $(v_i)_{i=1}^d$, we can assume that $k>0$ for
		otherwise $A=0$ and the assertion is trivial. Given a parameter
		$\sigma>0$, let $B$ be defined by
		\[
		B := \sum_{i=1}^k \sigma v_i \otimes v_i + \sum_{i=k+1}^d
		\sigma^{-\beta}  v_i \otimes v_i
		\]
		and $\beta=k/(d-k)$ because then $\det B = \sigma^k\sigma^{-\beta(d-k)}=1$.
		Therefore,
		\[
		AB = \sigma \sum_{i=1}^k
		\lambda_i v_i \otimes v_i
		\quad\Rightarrow\quad
		\tr(AB) = \sigma \sum_{i=1}^k \lambda_i
		\to 0
		\quad \textrm{as } \sigma \to 0,
		\]
		which proves (i) for $A$ singular since $B$ is SPD.
		
		\medskip
		{\it Step 3: Proof of (ii)}.
		Let $A$ and $B$ be SPSD matrices and $0\le \lambda \le 1$.   
		Then $\lambda A + (1-\lambda) B$ is also SPSD and we can apply (i) to 
		\begin{align*}
		(\det{[ \lambda A + (1- \lambda) B ]})^{1/d} &= \frac{1}{d} \inf{ \left\{ \tr[( \lambda A + (1- \lambda) B)C] \Big| \ C\  \textrm{is SPD and} \det{C}=1 \right\} }
		\\ &\geq \frac{\lambda}{d} \inf{ \left\{ \tr(AC) \ \Big| \ C\  { \rm is \ SPD \ and} \  \det{C}=1 \right\} } 
		\\&+ \frac{1- \lambda}{d} \inf{ \left\{  \tr(BC) \ \Big| \ C\  { \rm is \ SPD \ and} \ \det{C}=1 \right\} }
		\\ &= \lambda (\det{A})^{1/d} + (1-\lambda) (\det{B})^{1/d}.
		\end{align*}
		This completes the proof.
	\end{proof}

	Upon relabeling $\widehat{A}=\lambda A$ and
	$\widehat{B}=(1-\lambda)B$, which are still SPSD, we can write Lemma
	\ref{L:Concavity} (ii) as follows:
	\begin{equation}\label{E:concavity}
	(\det \widehat{A})^{1/d} + (\det \widehat{B})^{1/d}
	\le \big(\det(\widehat{A}+\widehat{B})\big)^{1/d}.
	\end{equation}
	We now show that our discrete operator $T_\varepsilon[\cdot]$
	possesses a similar property.

	\begin{Corollary} [concavity of discrete operator] \label{C:OperIneq}
		Given two functions $u_h,w_h\in\Vh$, we have
		\[
		\big( T_{\epsilon} [u_h] (x_i) \big)^{1/d} +
		\big( T_{\epsilon} [w_h](x_i) \big)^{1/d}  \le
		\big( T_{\epsilon}[u_h+w_h](x_i) \big)^{1/d},
		\] 
		for all nodes $x_i \in \Nhi$ such that
		$\sdd u_h(x_i;v_j)\ge0, \ \sdd w_h(x_i;v_j)\geq 0$ for all $v_j \in \St$.
	\end{Corollary}
	\begin{proof}
		We argue in two steps.
		
		\medskip
		\textit{Step 1.}
		For $a=(a_j)_{j=1}^d \in \mathbb{R}^d$ with $a_j \geq 0, \ j=1,\ldots,d$ we consider 
		the function
		\[
		f(a) := \left( \prod_{j=1}^d a_j \right)^{1/d},
		\]
		which can be conceived as the determinant of a
                diagonal (and thus symmetric)
		positive semi-definite matrix with diagonal elements $(a_j)_{j=1}^d$, i.e.
		\[
		f(a) = \big(\det { \rm diag
			}{\left\{a_1,\ldots,a_d\right\}}  \big)^{1/d}.
		\]
		Applying \eqref{E:concavity} to $\widehat{A}={ \rm diag
			}{\left\{a_1,\ldots,a_d\right\}} , \widehat{B}={ \rm diag
		}{\left\{b_1,\ldots,b_d\right\}} $
		with $a=(a_j)_{j=1}^d, b=(b_j)_{j=1}^d\ge0$ component wise, we deduce
		\[
		f(a) + f(b) \le f(a+b).
		\]

		{\it Step 2.}
		We now apply this formula to the discrete operator. Since both
		$u_h,w_h$ are discretely convex at $x_i\in\Nhi$, so is $u_h+w_h$, and we
		can apply Lemma \ref{L:DisConv} (discrete convexity) to write
		\[
		T_\varepsilon[u_h+w_h](x_i) = \prod_{j=1}^d \sdd[u_h+w_h](x_i;v_j)
		\]
		for a suitable $\mathbf{v}=(v_j)_{j=1}^d\in\Vperpt$. Making use again
		of \eqref{E:simpler-def}, this time for $u_h$ and $w_h$ and for the
		specific set of directions $\mathbf{v}$ just found, we obtain
		\begin{align*}
		\big( T_\varepsilon [u_h](x_i)  \big)^{\frac{1}{d}} &
		+ \big( T_\varepsilon [w_h](x_i)  \big)^{\frac{1}{d}}
		\le \left(\prod_{j=1}^d \sdd u_h(x_i;v_j)\right)^{\frac{1}{d}}
		+ \left(\prod_{j=1}^d \sdd w_h(x_i;v_j)\right)^{\frac{1}{d}}
		\\
		& \le  \left(\prod_{j=1}^d \sdd u_h(x_i;v_j) + \sdd w_h(x_i;v_j) \right)^{\frac{1}{d}}
		= \big(T_\varepsilon[u_h+w_h](x_i)\big)^{\frac{1}{d}},
		\end{align*}
		where the second inequality is given by Step 1 for
                  $a = ( \sdd u_h(x_i;v_j) )_{j=1}^d$ and
                  $b = ( \sdd w_h(x_i;v_j) )_{j=1}^d$.
		This is the asserted estimate.
	\end{proof}
	
	\subsection{Continuous Dependence of the Two-Scale Method on Data}
	\label{S:cont-depend}	
	
	We are now ready to prove the continuous dependence  of discrete solutions on data. This will be instrumental later for deriving rates of convergence for the two-scale method. 
	
	\begin{Proposition}[continuous dependence on data] \label{P:ContDep}
		Given two functions $u_h,w_h\in\Vh$ such that $u_h \geq w_h$ on $\partial \Omega_h$ and
		\[
			T_{\varepsilon}[u_h](x_i) =f_1(x_i) \geq 0 \ \ {\rm and} \  \
			T_{\varepsilon}[w_h](x_i) = f_2(x_i)  \geq 0
		\]
		at all interior nodes $x_i\in\Nhi$, we have that
		\[
		\max \limits_{ \Omega_h} (u_h-w_h)^- \leq C \ \delta \ \left(
		\sum\limits_{x_i \in \cC_{-}(u_h-w_h)} \left(
		f_1(x_i)^{1/d}- f_2(x_i)^{1/d}  \right)^d   \right)^{1/d}. \]
	\end{Proposition}
	\begin{proof}
		Since $u_h-w_h\in\Vh$ and $u_h -w_h \geq 0$ on $\partial \Omega_h$,
		\Cref{L:Stability} (stability) yields
		$$
		\max_{x_i\in\Nhi} (u_h - w_h) (x_i)^- \leq C \delta
		\left( \sum_{x_i \in \mathcal{C}_-(u_h - w_h)} 
		T_{\epsilon} [u_h - w_h] (x_i)  \right)^{1/d} .
		$$
		Since  $x_i \in \cC_-(u_h-w_h)$, we have that $\sdd (u_h-w_h)(x_i;v_j)\geq 0$, whence 
		$$
			\sdd u_h(x_i;v_j)\geq \sdd w_h(x_i;v_j) \ge 0
		\quad\forall v_j\in\St,
		$$
		where we have made use of Lemma \ref{L:DisConv} (discrete convexity).
		Invoking \Cref{C:OperIneq} (concavity of discrete operator)
		for $u_h-w_h$ and $w_h$, we deduce
		\[
		\big( T_{\epsilon} [u_h-w_h] (x_i) \big)^{1/d} \le
		\big( T_{\epsilon} [u_h](x_i) \big)^{1/d} -
		\big( T_{\epsilon} [w_h](x_i) \big)^{1/d},
		\]
		whence
		\begin{align*}
		\max_{x_i\in\Nhi} (u_h - w_h) (x_i)^- \leq &\;
		C \delta \left( \sum_{x_i \in \mathcal{C}_-(u_h - w_h)} 
		\left( T_{\epsilon} [u_h](x_i)^{1/d}  - T_{\epsilon} [w_h] (x_i)^{1/d} \right)^d
		\right)^{1/d} 
		\\
		= &\; C \delta \left( \sum_{x_i \in \mathcal{C}_-(u_h - w_h)} 
		\left(  f_1(x_i)^{1/d}  - f_2(x_i)^{1/d} \right)^d \right)^{1/d} .
		\end{align*}
		This completes the proof.
	\end{proof}

	\section{Rates of Convergence} \label {S:RoC}

	We now combine the preceding estimates to prove
        pointwise convergence rates for	solutions with varying degree of regularity. 
       We first present in \Cref{T:RatesHolder} the case of a classical solution with H\"older regularity. 
       This allows us to introduce the main techniques employed for deriving the rates of convergence. We then build on these techniques and prove error estimates for three more cases of increasing generality. In \Cref{T:RatesSobolev} we assume a classical solution with Sobolev regularity, which requires the use of embedding estimates and accumulating the truncation error in $l^d$, rather than $l^\infty$. We next deal with a non-classical solution that is globally in $W^{2}_{\infty}(\Omega)$ but its Hessian is discontinuous across a $d-1$ dimensional Lipschitz surface. To prove rates for this case we need to take advantage of the small volume affected by this discontinuity and combine it with the techniques used in \Cref{T:RatesHolder} and \Cref{T:RatesSobolev}. Lastly, we remove the non-degeneracy assumption $f \geq f_0 >0$ used in the previous three cases to obtain rates of convergence for a piecewise smooth viscosity solution with degenerate right hand side $f$. This corresponds to one of the numerical experiments performed in \cite{NoNtZh}.
        Our estimates do not require $h$ small and are
          stated over the computational domain
          $\Omega_h\subset\Omega$.
	
	\subsection{Barrier Function}\label{S:Barrier}
	We recall here the two discrete barrier functions
        introduced in \cite[Lemmas 5.1, 5.2]{NoNtZh}.
	The first one is critical in order to control the behavior of
        $\uve$ close to the boundary of $\Omega_h$ and prove the
	convergence to the unique viscosity solution $u$ of \eqref{E:MA}. We now 
	use the same barrier function to control the pointwise error of $\uve$ and $u$
	close to the boundary. The second barrier allows us
        to treat the degenerate case
	$f \geq 0$, using techniques similar to the case $f >0$.
	
	\begin{Lemma}[discrete boundary barrier] \label{L:Barrier} Let $\Omega$ be uniformly convex and  $E>0$ be arbitrary. For each node $z \in \Nhi$ with ${\rm dist}(z,\partial \Omega_h) \leq \delta$, there exists a function $p_h\in\Vh$ such that $T_{\varepsilon}[p_h](x_i) \geq E$ for all $x_i \in \Nhi$, $p_h \leq 0$ on $\partial \Omega_h$ and
		\[
		|p_h(z)| \leq CE^{1/d} \delta 
		\]
		with $C$ depending on $\Omega$.
	\end{Lemma}
	
	\begin{Lemma}[discrete interior barrier] \label{L:BarrierInterior}
		Let $\Omega$ be contained in the ball $B(x_0,R)$ of center $x_0$ and radius
		$R$. If $q(x):= \frac12\big( |x-x_0|^2 - R^2 \big)$, then its
		interpolant $q_h:=\interp q\in\Vh$ satisfies
		\[
		T_\ve[q_h](x_i) \ge 1\quad\forall x_i\in\Nhi,
		\qquad
		q_h(x_i) \le 0 \quad\forall x_i\in\Nhb.
		\]
	\end{Lemma}	
	
	\subsection{Error Estimates for Solutions with H\"older Regularity}\label{S:RatesHolder}
	
	We now deal with classical solutions $u$ of (\ref{E:MA}) of
        class $C^{2+k,\alpha}(\overline{\Omega})$, with $k=0,1$ and
        $0<\alpha \leq 1$, and derive pointwise error estimates.
	We proceed as follows. We first use \Cref{L:Barrier}
          (discrete boundary barrier) to control $\uve -\interp u$ in
          the $\delta$-neighborhood $\omega_{h,\delta}$ 
          of $\partial \Omega_h$, where the consistency error of
          $T_\varepsilon[\interp u]$ is of order one according to
          \Cref{L:FullConsistency} (consistency of $T_\varepsilon[\interp u]$).
          In the $\delta$-interior region $\Omega_{h,\delta}$ we combine
          the interior consistency error of
          $T_\varepsilon[\interp u]$ from \Cref{L:FullConsistency} and \Cref{P:ContDep}
          (continuous dependence on data). Judicious choices of
          $\delta$ and $\theta$ in terms of $h$ conclude the argument.
	
	\begin{Theorem}[rates of convergence for classical solutions]\label{T:RatesHolder}
		Let $f(x) \geq f_0 >0$ for all $x \in \Omega$. Let $u$ be the
		classical solution of (\ref{E:MA}) and $\uve$ be the discrete
		solution of (\ref{E:2ScOp}). If $u \in \Ctao$ for $0<\alpha
		\leq 1$ and
                \[
		\delta = R_0(u) \ h^{\frac{2}{2+\alpha}},
		\quad
		\theta = R_0(u)^{-1} \ h^{\frac{\alpha}{2+\alpha}}
		\]
		with $R_0(u) = |u|_{\Wtio}^{\frac{1}{2+\alpha}} \ |u|_{\Ctao}^{-{\frac{1}{2+\alpha}}}$, then
		\[
		\|u-\uve\|_{L^{\infty}(\Omega_h)} \leq C(\Omega,d,f_0)
                \Big( |u|_{\Ctao}^{\frac{1}{2+\alpha}}
                \ |u|_{\Wtio}^{d-\frac{1}{2+\alpha}} + \big(1+ R_0(u)\big) \
                |u|_{\Wtio} \Big) \ h^{\frac{\alpha}{2+\alpha}}.
		\]
		Otherwise, if $u \in \Cthao$ for $0 < \alpha \leq 1$ and
		\[
		\delta = R_1(u) \ h^{\frac{2}{3+\alpha}},
		\quad
		\theta = R_1(u)^{-1} \ h^{\frac{1+\alpha}{3+\alpha}}
		\]
		with $R_1(u) := |u|_{\Wti(\Omega)}^\frac{1}{3+\alpha}|u|_{\Cthao}^{-\frac{1}{3+\alpha}}$, then
		\[
		\|u-\uve\|_{L^{\infty}(\Omega_h)} \leq  C(\Omega,d,f_0)
                \Big( |u|_{\Cthao}^{\frac{1}{3+\alpha}}
                \ |u|_{\Wtio}^{d-\frac{1}{3+\alpha}} +
                \big(1 + R_1(u) \big) \
                |u|_{\Wtio} \Big) \ h^{\frac{1+\alpha}{3+\alpha}}.		
                \]
	\end{Theorem}
	
\begin{proof}
If $R_k(u) :=
|u|_{\Wti(\Omega)}^\frac{1}{2+k+\alpha}|u|_{\Cka(\overline{\Omega})}^{-\frac{1}{2+k+\alpha}}$,
$k=0,1$, we prove below the estimate
\[
\max_{\Omega_h} \, (\uve - \interp u) \lesssim
\Big( \big(1+R_k(u)|u|_{\Wtio}\big)+ |u|_{\Cka(\overline{\Omega})}^{\frac{1}{2+k+\alpha}}
\ |u|_{\Wtio}^{d-\frac{1}{2+k+\alpha}}
\Big) \ h^{\frac{k+\alpha}{2+k+\alpha}}
\]
with a hidden constant depending on $\Omega, d, f_0$.
We proceed in three steps.
The estimates for $\max_{\Omega_h} \, (\interp u - \uve)$ are similar         
and thus omitted. Adding the interpolation error $\|u-\interp
u\|_{L^{\infty}(\Omega_h)} \leq Ch^2 |u|_{\Wtio}$
\cite{BrenScott} readily gives the asserted estimates because
$\frac{k+\alpha}{2+k+\alpha} \le \frac{1}{2}$ for $k=0,1$ and
$0<\alpha\le 1$.

		\medskip        
		\textit{Step 1: Boundary estimate.} We show
		that for $z \in \Nhi$ so that
                ${\rm dist}(z,\partial \Omega_h) \leq \delta$
		$$
		\uve(z) - \interp u(z) \leq C |u|_{\Wtio} \delta.
		$$
		Given the function $p_h$ of \Cref{L:Barrier}
		(discrete boundary barrier), for $z$ fixed, we examine the behavior of $\uve + p_h$. For any interior node $x_i \in \Nhi$, we have
		$$
		\begin{aligned}
		\prod_{j=1}^d \sdd (\uve +p_h) (x_i;v_j) &= \prod_{j=1}^d (\sdd \uve (x_i;v_j)  + \sdd p_h(x_i;v_j)  )\\
		&\geq \prod_{j=1}^d \sdd \uve (x_i;v_j) + \prod_{j=1}^d \sdd  p_h (x_i;v_j) \quad \forall \mathbf{v}=(v_j)_{j=1}^d \in \Vperpt,
		\end{aligned}
		$$
		because $\sdd \uve (x_i;v_j) \geq 0$ and $\sdd p_h(x_i;v_j)\geq 0$. We apply \Cref{L:FullConsistency} (consistency of $T_{\varepsilon}[\interp u])$ to obtain
		$$
		\begin{aligned}
		T_{\varepsilon} [\uve + p_h] (x_i) & \geq T_{\varepsilon}[\uve](x_i) + T_{\varepsilon}[p_h](x_i) \\
		& \geq f(x_i) + E \\
		& \geq T_{\varepsilon}[\interp u](x_i) - C|u|_{\Wtio}^d +  E \geq T_{\varepsilon} [\interp u](x_i),
		\end{aligned}
		$$
		provided $E \geq C|u|_{\Wtio}^d$. Since $\interp u = \uve$ and
		$p_h \leq 0 $ on $\partial \Omega_h$, we deduce from
		\Cref{L:DCP} (discrete comparison principle) that
		$$
		\uve(z) + p_h(z) \leq \interp u (z),
		$$
		whence,
		$$
		\uve(z) - \interp u(z) \leq C |u|_{\Wtio} \delta.
		$$
	
		\medskip
		\textit{Step 2: Interior estimate}. We show that for
                all $x_i \in \Nhi$ so that
                ${\rm dist}(x_i,\partial \Omega_h) \geq \delta$
		$$
		T_{\varepsilon} [\uve] (x_i)- T_{\varepsilon} [\interp u](x_i) \leq C_1(u) \delta^{\alpha +k} + C_2(u) \left( \frac{h^2}{\delta^2}+\theta^2 \right)
		$$
		with $k=0,1$ and 
		$$
		C_1(u) = C |u|_{\Cka(\overline{\Omega})} \ |u|_{\Wtio}^{d-1}, \quad C_2(u) = C |u|_{\Wtio}^d
		$$
        dictated by \Cref{L:FullConsistency}.
		Step 1 guarantees that 
		$$
		\uve - \interp u \leq C|u|_{\Wtio} \delta \quad \text{on }  \partial \Omega_{h,\delta},
		$$
		where $\Omega_{h,\delta}$ is defined in \eqref{Omega-delta}.
		Let $\dep:= \interp u-\uve +C|u|_{\Wtio}\delta$ and note that $\dep \geq 0 $ on $\partial \Omega_{h,\delta}$. We then apply \Cref{P:ContDep} (continuous dependence on data) to $\dep$ in $\Omega_{h,\delta}$, in conjunction with \Cref{L:FullConsistency} (consistency of $T_{\varepsilon}[\interp u]$), to obtain 
		$$
		\max_{\Omega_{h,\delta}} \dep^- \lesssim
                \ \delta \left(  \sum_{x_i \in \cC_-(\dep)} \left( (f(x_i) +e)^{1/d} - f(x_i)^{1/d}  \right)^d  \right)^{1/d}
		$$
		with $e:= C_1(u) \delta^{\alpha+k} + C_2(u) \left(
                \frac{h^2}{\delta^2} + \theta^2 \right)$. We now use
                that the function $t \mapsto t^{1/d}$ is concave with derivative
                $\frac{1}{d}t^{1/d-1}$ and $f(x_i)\geq f_0 >0$ to estimate
		$$
		(f(x_i)+e)^{1/d}-f(x_i)^{1/d} \leq \frac{e}{d f_0^{\frac{d-1}{d}}},
		$$	
		whence 
		$$
		\max_{\Omega_{h,\delta}} \dep^- \lesssim \delta \left(  \sum_{x_i \in \cC_-(\dep)} \left(C_1(u) \delta^{\alpha +k}+C_2(u)\left(\frac{h^2}{\delta^2}+\theta^2\right) \right)^d \right)^{1/d}.
		$$
		Since the cardinality of $\cC_-(\dep)$ is bounded by that of $\mathcal N_h$, which in turn is bounded by $Ch^{-d}$ with $C$ depending on shape regularity, we end up with 
		\begin{equation} \label{E:rates-factor}
		\max_{\Omega_h} \, (\uve -\interp u) \lesssim
                |u|_{\Wtio} \delta  + \frac{\delta}{h} \left( C_1(u)
                \delta^{\alpha+k} + C_2(u) \Big( \frac{h^2}{\delta^2}
                + \theta^2  \Big) \right).
		\end{equation}
		
		\medskip
		\textit{Step 3: Choice of $\delta$ and
                    $\theta$.} To find an optimal choice of $\delta$
                and $\theta$ in terms of $h$, we minimize the
                right-hand side of the preceding estimate. We first
                set $\theta^2 = \frac{h^2}{\delta^2}$ and equate
                the last two terms
		$$
		C_1(u) \delta^{k+\alpha} = C_2(u) \frac{h^2}{\delta^2} \quad \Longrightarrow \quad \delta = R_k(u) h^{\frac{2}{2+k+\alpha}}.
		$$
		Writing again $C_1(u)$ and
                $C_2(u)$ in terms of $|u|_{\Cka(\overline{\Omega})}$
                and $|u|_{\Wtio}$, we thus obtain
		$$
		\max_{\Omega_h} \, (\uve - \interp u) \lesssim
                R_k(u) |u|_{\Wtio} \ h^{\frac{2}{2+k+\alpha}}
                + |u|_{\Cka(\overline{\Omega})}^{\frac{1}{2+k+\alpha}}
                \ |u|_{\Wtio}^{d-\frac{1}{2+k+\alpha}}
                \ h^{\frac{k+\alpha}{2+k+\alpha}}.
		$$
		Finally, the desired estimate follows immediately because
                $k+\alpha \leq 2$.
	\end{proof}
		We observe that according to \Cref{T:RatesHolder} the rate of convergence is of order $h^{1/2}$ whenever $u\in C^{3,1}(\overline{\Omega})$. However, our numerical experiments in \cite{NoNtZh} indicate linear rates of convergence, which correspond to \Cref{L:FullConsistency} (consistency of $T_\varepsilon[\interp u_h]$). This mismatch may be attributed to the factor $\frac{\delta}{h} \gg 1$ in \eqref{E:rates-factor}, which relates to \Cref{R:VolumeSubopt} (artificial factor $\frac{\delta}{h}$). This issue will be tackled in a forthcoming paper. 

	\subsection{Error Estimates for Solutions with Sobolev Regularity}\label{S:RatesSobolev}
	
	We now derive error estimates for solutions $u \in W^s_p(\Omega)$ with $s>2+\frac{d}{p}$ so that $W^s_p(\Omega)\subset C^2(\overline{\Omega})$.
	We exploit the structure of the estimate of \Cref{P:ContDep} (continuous dependence on data) which shows that its right-hand side accumulates in $l^d$ rather than $l^{\infty}$.
	\begin{Theorem} [convergence rate for $W_p^s$
		solutions] \label{T:RatesSobolev} Let $f \geq f_0 >0$ in $\Omega$
		and let the viscosity solution $u$ of (\ref{E:MA}) be of class
		$W_p^s(\Omega)$ with $\frac{d}{p} < s-2-k\le 1,
		\ k=0,1$. If $\uve$ is the discrete solution of
                (\ref{E:2ScOp}) and
		\[
		\delta = R(u) \ h^{\frac{2}{s}},
		\quad
		\theta = R(u)^{-1} \ h^{1-\frac{2}{s}},
		\]
                with $R(u) := |u|_{\Wtio}^{\frac{1}{s}} |u|_{\Wspo}^{-\frac{1}{s}}$,
		then
		$$
                \| u - \uve\|_{L^{\infty}(\Omega_h)} \leq
                C(d,\Omega,f_0)  \Big( |u|_{\Wspo}^{\frac{1}{s}}
                \ |u|_{\Wtio}^{d-\frac{1}{s}} + \big(1 + R(u) \big)
                |u|_{W^2_\infty(\Omega)} \Big) h^{1-\frac{2}{s}} .
		$$
	\end{Theorem}
	\begin{proof}
	  We proceed as in \Cref{T:RatesHolder}
          to show an upper bound for $\uve - \interp u$.
          The boundary estimate of Step 1 remains intact, namely
		$$
		\uve(z) - \interp u(z) \leq C \ |u|_{\Wtio} \ \delta
		$$	
		for all $z \in \Nhi$ such that
                ${\rm dist}(z,\partial \Omega_h) \leq \delta$.
                On the other hand, Step 2 yields
		$$
		\max_{\Omega_{h,\delta}}(\uve - \interp u)
                \lesssim  \delta |u|_{\Wtio} + \delta \left(
                \sum_{x_i \in \Nhi} C_1(u)^d  \delta^{(k+\alpha)d} +
                C_2(u)^d \left( \frac{h^2}{\delta^2} + \theta^2
                \right)^d  \right)^{1/d},
		$$
		where $C_1(u)$ and $C_2(u)$ are defined in \Cref{L:FullConsistency}
		(consistency of $T_\varepsilon [\interp u]$) and $0 <
                \alpha = s-2-k-\frac{d}{p} \leq 1$ corresponds to the
                Sobolev embedding $W_p^s(B_i) \subset \Cka(B_i)$.
                In the following calculations we resort to the
                Sobolev inequality \cite[Theorem 2.9]{Giusti}
                $$  
                  |u|_{\Cka(B_i)} \le C |u|_{W_p^s(B_i)},
                $$
                involving only semi-norms. We stress that $C>0$
                depends on the Lipschitz constant of $B_i$ but not on
                its size. The latter is due to the fact that the
                Sobolev numbers of $W^{s-2-k}_p(B_i)$ and
                $C^{0,\alpha}(B_i)$ coincide:
                $0< s-k-2-d/p=\alpha \le 1$. We refer to \cite[Theorem
                2.9]{Giusti} for a proof for $0<s<1$.
                We now use the H\"older
                inequality with exponent $\frac{p}{d} > 1$ to obtain
		$$
		\begin{aligned}
		\left( \sum_{x_i \in \Nhi} C_1(u)^d \right)^{\frac{1}{d}} &\lesssim \left(  \sum_{x_i \in \Nhi}  |u|_{W_p^s(B_i)}^d |u|_{\Wti(B_i)}^{d(d-1)} \right)^{\frac{1}{d}} \\
		&\lesssim \left( \sum_{x_i \in \Nhi} |u|_{W_p^s(B_i)}^{d\frac{p}{d}}  \right)^{\frac{1}{d}\frac{d}{p}} \ \left( \sum_{x_i \in \Nhi} |u|_{\Wti(B_i)}^{d(d-1)\frac{p}{p-d}}  \right)^{\frac{1}{d}\frac{p-d}{p}}.
		\end{aligned}
		$$
		Since the cardinality of the set of balls $B_i$ containing an
		arbitrarily given $x \in \Omega$ is proportional to $\left(
		\frac{\delta}{h} \right)^d$, while the cardinality of $\Nhi$ is
		proportional to $h^{-d}$,
		we get
		\begin{align*}
		\left( \sum_{x_i \in \Nhi} C_1(u)^d \right)^{\frac{1}{d}} 
		& \lesssim \left( \frac{\delta}{h} \right)^{\frac{d}{p}} |u|_{\Wspo}
		\ \left( h^{-d} |u|_{\Wtio}^{\frac{d(d-1)p}{p-d}}
		\right)^{\frac{p-d}{pd}} \\
		& \lesssim \frac{\delta^{\frac{d}{p}}}{h} \ |u|_{\Wspo} \ |u|_{\Wtio}^{d-1}.
		\end{align*}
		Exploiting that $\alpha + k + \frac{d}{p}+1=s-1$,
                we readily arrive at
		$$
		\delta \left( \sum_{x_i \in \Nhi} C_1(u)^d \ \delta^{(k+\alpha)d}
                \right)^{\frac{1}{d}} \lesssim \frac{\delta^{s-1}}{h} |u|_{\Wspo} \ |u|_{\Wtio}^{d-1}.
		$$
		In addition, we have 
		$$
		\left( \sum_{x_i \in \Nhi} C_2(u)^d \right)^{\frac{1}{d}}
                \lesssim  |u|_{\Wtio}^d \ \frac{1}{h},
		$$
		whence
		$$
		\delta \left( \sum_{x_i \in \Nhi} C_2(u)^d \left( \frac{h^2}{\delta^2} + \theta^2 \right)^d  \right)^{\frac{1}{d}} \lesssim |u|_{\Wtio}^d \  \frac{\delta}{h} \left( \frac{h^2}{\delta^2} + \theta^2 \right).
		$$
		Collecting the previous estimates, we end up with
		$$
		\max_{\Omega_h} \, (\uve - \interp u) \lesssim
                \delta |u|_{\Wtio} + |u|_{\Wtio}^{d-1}
                \frac{\delta}{h} \left(  |u|_{\Wspo} \delta^{s-2} +
                |u|_{\Wtio}  \left( \frac{h^2}{\delta^2} + \theta^2 \right) \right).
		$$
		To find an optimal relation among $h,\delta$ and
                $\theta$, we first choose $\theta^2 =
                \frac{h^2}{\delta^2}$ and next equate the two terms in
                the second summand to obtain
		$$
		\delta = R(u) \ h^{\frac{2}{s}},
		\quad
                \theta = R(u)^{-1} \ h^{1-\frac{2}{s}},
		$$
		whence
		$$
		\max_{\Omega_h} \, (\uve - \interp u) \lesssim
                R(u) |u|_{\Wtio} h^{\frac{2}{s}} + 
                |u|_{\Wspo}^{\frac{1}{s}} \ |u|_{\Wtio}^{d-\frac{1}{s}}
                h^{1-\frac{2}{s}}.
		$$
 Adding the interpolation error estimate $\|u-\interp u\|_{L^\infty(\Omega)}
 \lesssim h^2 |u|_{\Wtio}$, and using that $2 > \frac{2}{s} \ge 1 -
 \frac{2}{s}$ for $2<s\le 4$, leads to the asserted estimate.
	\end{proof}
	
	The error estimate of \Cref{T:RatesSobolev} (convergence rate for $W_p^s$-solutions) is of order $\frac{1}{2}$ for $s=4$ and $u \in W_p^4(\Omega)$ with $p > d$. This rate requires much weaker regularity than the corresponding error estimate in \Cref{T:RatesHolder}, namely $u \in C^{3,1}(\overline{\Omega}) = W^4_{\infty}(\Omega)$. In both cases, the relation between $\delta$ and $h$ is $\delta \approx h^{\frac{1}{2}}$.

	\subsection{Error Estimates for Piecewise Smooth Solutions} \label{S:RatesPW}
	
	We now derive pointwise rates of convergence for a larger class	of solutions than in \Cref{S:RatesSobolev}. These are viscosity
	solutions which are piecewise $W_p^s$ but have discontinuous Hessians
	across a Lipschitz $(d-1)$-dimensional manifold $\cS$;
	we refer to the second numerical example in \cite{NoNtZh}.
	Since $T_\varepsilon[\interp u]$ has a
          consistency error of order one in a $\delta$-region around
          $\cS$, due to the discontinuity of $D^2u$, we 
          exploit the fact that the measure of this region is
          proportional to  $\delta|\cS|$. We are thus able to adapt the argument of
          \Cref{T:RatesSobolev} (convergence rate for $W^s_p$ solutions),
          and accumulate such consistency error in $l^d$,
          at the expense of an extra additive term of order
          $h^{-1}\delta^{1+\frac{1}{d}}$. This term is responsible for
          a reduced convergence rate when $u \in W^s_p(\Omega\setminus\cS)$,
          $s > 2+ \frac{1}{d}$.
	\begin{Theorem}[convergence rate for piecewise smooth
            solutions] \label{T:RatesPW}
        Let $\cS$ denote a $(d-1)$-dimensional Lipschitz manifold
        that divides $\Omega$ into two disjoint subdomains $\Omega_1, \Omega_2$
        so that $S = \overline{\Omega}_1 \cap \overline{\Omega}_2$.
Let $f \geq f_0>0$ in $\Omega$ and let $u \in W_p^s(\Omega_i) \cap
\Wtio$, for $i=1,2$ and $\frac{d}{p} < s-2-k \le 1, k=0,1$,
be the viscosity solution of (\ref{E:MA}). If $\uve$ denotes the
discrete solution of (\ref{E:2ScOp}), then for $\beta = \min\{s,2+\frac{1}{d}\}$
we have
                $$
                \| u -\uve\|_{L^{\infty}(\Omega_h)} \leq
                C(d,\Omega,f_0) \left( R(u)^{-1} |u|_{\Wtio}^d +
                \big(1+ R(u) \big) |u|_{\Wtio}
                 \right) h^{1-\frac{2}{\beta}},
		$$
		with
                $R(u) = \Big(\frac{|u|_{\Wtio}}{|u|_{W_p^s(\Omega\setminus \cS)}+{|u|_{\Wtio}}}\Big)^{\frac{1}{\beta}}$
                and $|u|_{W_p^s(\Omega \setminus \cS)}:= \max_i |u|_{W_p^s(\Omega_i)}$,
                provided
		\[
		\delta = R(u) \ h^{\frac{2}{\beta}},
		\quad
		\theta= R(u)^{-1} \ h^{1-\frac{2}{\beta}}.
		\]
	\end{Theorem}
	\begin{proof}
	  We proceed as in Theorems \ref{T:RatesHolder} and
          \ref{T:RatesSobolev}. The
		boundary layer estimate relies on the regularity $u \in \Wtio$ which
		is still valid, whence for all $x \in \Omega_h$
                such that ${\rm dist}(x,\partial \Omega_h) \leq \delta$ 
                we obtain
                $$
		\uve(x) - \interp u(x) \leq C |u|_{\Wtio} \delta.
		$$
		Consider now the internal layer
		\begin{equation*}
		\cS^{\delta}_h : = \left\{  x \in \Omega_h : \ {\rm dist}(x,\cS) \leq \delta  \right\},
		\end{equation*}
		which is the region affected by the discontinuity of
                the Hessian $D^2u$. Recall the auxiliary function
                $\dep = \interp u - \uve + C |u|_{\Wtio} \delta$
                of \Cref{T:RatesHolder} (rates of convergence for
                classical solutions) and split the contact set $\cC_-^\delta(\dep):= \cC_-(\dep) \cap \Omega_{h,\delta}$ as follows:
		\begin{equation*}
		\cS_{h,1}^{\delta} := \cC_-^\delta(\dep) \cap \cS^{\delta}_h,
		\quad \cS_{h,2}^{\delta}:= \cC_-^\delta(\dep) \setminus \cS^{\delta}_h.
		\end{equation*}
                An argument similar to Step 2 (interior estimate)
                of \Cref{T:RatesHolder}, based on combining
		\Cref{P:ContDep} (continuous dependence on data)
		and \Cref{L:FullConsistency} (consistency of $T_\varepsilon[\interp u]$)
                with assumption $f\ge f_0>0$, yields
		$$
		\begin{aligned}
		\max_{\Omega_{h,\delta}} \ \dep^- &\lesssim
                \delta \left(  \sum_{x_i \in \cS_{h,1}^{\delta}} C_2(u)^d
		\right)^{1/d}
		\\ & + \delta \ \left(  \sum_{x_i \in \cS_{h,2}^{\delta}} C_1(u)^d \delta^{(k+\alpha)d} + C_2(u)^d \left(  \frac{h^2}{\delta^2} + \theta^2  \right)^d \right)^{1/d}
		=: I_1 + I_2,
		\end{aligned}
		$$
		because the consistency error in $\cS_{h,1}^{\delta}$ is bounded by
		$C_2(u) = C |u|_{W^2_\infty(B_i)}^d$. As in
                \Cref{T:RatesSobolev} (convergence rate for $W^s_p$ solutions),
		$C_1(u)$ satisfies
                \[
                C_1(u) \lesssim |u|_{W^s_p(B_i)} |u|_{W^2_\infty(B_i)}^{d-1}.
                \]
		Since the number of nodes $x_i \in \cS_{h,1}^{\delta}$ is bounded by $C|\cS|\delta h^{-d}$, we deduce
		$$
		I_1 \lesssim \delta \left( \sum_{x_i \in \cS_{h,1}^{\delta}} C_2(u)^d  \right)^{1/d} \lesssim \ |u|_{\Wtio}^d \frac{\delta^{1+\frac{1}{d}}}{h}.
		$$
		For $I_2$ we distinguish whether $x_i$ belongs to
                $\Omega_1$ or $\Omega_2$ and accumulate $C_1(u)$ in
                $\ell^p$, exactly as in \Cref{T:RatesSobolev}, to obtain
		$$
		I_2 \lesssim \ |u|_{\Wtio}^{d-1} \left(  |u|_{W_p^s(\Omega \setminus \cS ) }\frac{\delta^{s-1}}{h} + |u|_{\Wtio} \ \frac{\delta}{h} \ \left( \frac{h^2}{\delta^2} + \theta^2  \right)  \right).
		$$
		Collecting the previous estimates and using the
                definition of $\beta$ yields 
		\begin{align*}
		\max_{\Omega_h}  (\uve &- \interp u) \lesssim |u|_{\Wtio} \delta
		\\
		& + \ |u|_{\Wtio}^{d-1} \frac{\delta}{h} \left(
		\Big(|u|_{W_p^s(\Omega \setminus \cS )}
                + |u|_{\Wtio}\Big)\delta^{\beta-2}  + |u|_{\Wtio}
		\ \left( \frac{h^2}{\delta^2} + \theta^2 \right)  \right).
		\end{align*}
		We finally realize that this estimate is similar
                to that in the proof of \Cref{T:RatesSobolev} except for the
                middle term on the right-hand side. Therefore, we
                proceed as in \Cref{T:RatesSobolev} to find the
                relation between $\delta, \theta$ and $h$, add the
                estimate $\|u-\interp u\|_{L^\infty(\Omega)}\lesssim h^2 |u|_{\Wtio}$,
                and eventually derive the asserted error estimate.
	\end{proof}

			\subsection{Error Estimates for Piecewise Smooth Solutions with Degenerate $f$} \label{S:RatesDegen}
	
			We observe that in all three preceding theorems we assume that $f \geq f_0 >0$.
			This is an important assumption in the proofs, since it allows us to use
			the concavity of $t\mapsto t^{1/d}$ and \Cref{P:ContDep} (continuous dependence on data) 
			to obtain 
			\begin{equation}\label{E:fconcavity}
			(f(x_i)+e)^{1/d} -f(x_i)^{1/d} \leq \frac{e}{df_0^{\frac{d-1}{d}}}, 
			\end{equation}
			where $e$ is related to the consistency of the operator in
			\Cref{L:FullConsistency} (consistency of $T_{\varepsilon}[\interp u]$).
			We see that this is only possible if $f_0 >0$. 
			If we allow $f$ to touch zero, then \eqref{E:fconcavity}
			reduces to 
			\begin{equation}\label{E:fconcavitydegen}
			(f(x_i)+e)^{1/d} -f(x_i)^{1/d} \leq e^{1/d}, 
			\end{equation}
			with equality for $f(x_i)=0$. This leads to a
                        rate of order $\big(\frac{\delta}{h}\big)^{1-\frac{2}{d}}\ge1$
                        for $d \geq 2$.
			To circumvent this obstruction, we use \Cref{L:BarrierInterior}
                        (interior barrier function) which allows us
			to introduce an extra parameter
                        $\sigma>0$ that
compensates for the lack of lower bound $f_0>0$ and yields
			pointwise error estimates of reduced order.
			
\begin{Theorem}[degenerate forcing $f\ge0$]\label{T:RatesDegen}
Let $\cS$  denote a $(d-1)$-dimensional
Lipschitz manifold that divides $\Omega$ into
two disjoint subdomains $\Omega_1, \Omega_2$ such that
$\cS = \overline{\Omega}_1 \cap \overline{\Omega}_2$.  Let $f \geq 0$ in
$\Omega$ and let $u \in W_p^s(\Omega_i) \cap \Wtio$, for $i=1,2$
and $\frac{d}{p} < s-2-k \le 1, k=0,1$, be the viscosity solution
of (\ref{E:MA}). If $\uve$ denotes the discrete solution of
(\ref{E:2ScOp}), then for $\beta = \min\{ s, 2+\frac{1}{d}\}$ we have
$$
\| u -\uve\|_{L^{\infty}(\Omega_h)} \leq C(d,\Omega)
|u|_{\Wtio} \Big( 1+ R(u) + R(u)^{-\frac{1}{d}} \Big)
 \ h^{\frac{1}{d}\left(1-\frac{2}{\beta}\right)}
$$
   with
$R(u) = \Big(\frac{|u|_{\Wtio}}{|u|_{W_p^s(\Omega\setminus \cS)}+|u|_{\Wtio}}\Big)^{\frac{1}{\beta}}$
and
$|u|_{W_p^s(\Omega \setminus \cS)}:= \max_i |u|_{W_p^s(\Omega_i)}$,
provided
\[
\delta = R(u) \ h^{\frac{2}{\beta}},
\quad
\theta= R(u)^{-1} \ h^{1-\frac{2}{\beta}}.
\]
			\end{Theorem}
			\begin{proof}
			We employ the interior barrier function $q_h$
                        of \Cref{L:BarrierInterior} scaled by
                        a parameter  $\sigma>0$  to control
                        $\uve-\interp u$ and $\interp u -\uve$ in
                        two steps. The parameter $\sigma$
                        allows us to mimic the calculation
                        in \eqref{E:fconcavity}. In the third step we
                        choose $\sigma$ optimally with respect to
                        the scales of our scheme.
	
			\medskip
			\textit{Step 1: Upper bound for
                            $\uve-\interp u$.} We let $w_h:= \uve + \sigma q_h$
                        and $v_h:= \interp u +C|u|_{\Wtio}
                        \delta$, observe that $T_\varepsilon[w_h](x_i)
                        \ge f(x_i) + \sigma^d$, and proceed as in Step
                        1 of \Cref{T:RatesHolder} to show
			$
			w_h(z) \leq v_h(z)
			$
                        for all $z\in\Nhi$ such that
                        $\textrm{dist} (z,\partial\Omega_h) \le \delta$.
                       
We now focus on $\Omega_{h,\delta}$ and 
define the auxiliary function $\dep:= v_h -w_h$ and contact set 
$\cC_-^\delta(\dep) := \cC_-(\dep) \cap \Omega_{h,\delta}$.
Since the previous argument guarantees that $\dep \geq 0$ on $\partial \Omega_{h,\delta}$, 
			\Cref{P:ContDep} (continuous dependence on data) gives
			$$
			\max_{\Omega_{h,\delta}} \dep^- \lesssim \
                        \delta \left( \sum_{x_i \in
                          \cC_-^\delta(\dep)}  \left(
                        \big(T_\varepsilon[v_h](x_i)\big)^{1/d}-
                        \big(T_\varepsilon[w_h](x_i)\big)^{1/d}
                        \right)^d   \right)^{1/d} .
                        $$
If $e_i$ is the local
consistency error given in \Cref{L:FullConsistency}, we further note that
$$
T_\varepsilon[v_h](x_i) \le f(x_i) + e_i,
\quad
T_\varepsilon[w_h](x_i) \geq T_\varepsilon[\uve](x_i) + T_\varepsilon[\sigma q_h](x_i) \geq f(x_i)+ \sigma^d
$$
for all $x_i\in\Nhi$, whence
$$
\begin{aligned}
\max_{\Omega_{h,\delta}} \dep^- &\lesssim \delta \left( \sum_{x_i \in \cC_-^\delta(\dep)}  \left( \big(f(x_i)+e_i\big)^{1/d}- \big(f(x_i)+\sigma^d\big)^{1/d} 
			\right)^d   \right)^{1/d}.
\end{aligned}
$$
We now observe that $e_i \geq \sigma^d$ for all
$x_i \in \cC_-^\delta(\dep)$
because all terms in the above sum are non-negative.
 If there is no such $x_i$, then the above bound implies that $\dep^-=0$
 and $w_h \leq v_h$, whence $\uve - \interp u \lesssim \sigma +
 |u|_{\Wtio} \delta$.
 Otherwise, the above observation combined with \eqref{E:fconcavity}
 and $f(x_i)\ge0$ implies
$$
\begin{aligned}
\big(f(x_i)+e_i \big)^{1/d} &- \big(f(x_i)+\sigma^d \big)^{1/d} 
\\ &=
\big(f(x_i)+\sigma^d+(e_i-\sigma^d)\big)^{1/d}- \big(f(x_i)+\sigma^d\big)^{1/d} \\
			&\leq \frac{e_i-\sigma^d}{d \sigma^{d\frac{d-1}{d}}} \leq d^{-1} \sigma^{1-d} e_i.
\end{aligned}
$$
We next proceed exactly as in \Cref{T:RatesPW}
(convergence rate for piecewise smooth solutions) to derive an upper
bound for $\dep^-$, but with the additional factor $\sigma^{1-d}$.
Employing the definition of $\dep$, we thereby obtain
$$
\uve - \interp u \lesssim \sigma  + |u|_{\Wtio} \delta +  \sigma^{1-d} \ \frac{\delta}{h} \left(  C_1(u) \delta^{s-2} + C_2(u) \  \left( \delta^{1/d} +  \frac{h^2}{\delta^2} + \theta^2  \right)  \right),
$$
where $C_1(u) = C |u|_{W_p^s(\Omega  \setminus \cS)} |u|_{\Wtio}^{d-1}$ and
$C_2(u)=C|u|_{\Wtio}^d$.

			 \medskip
			 \textit{Step 2: Lower bound for $\uve-\interp u$.}
			 To prove the reverse inequality, we
                         proceed as in Step 1, except that this time we 
			 define $v_h:= \uve + C |u|_{\Wtio} \delta$
                         and $w_h:= \interp u + \sigma q_h$. An
                         argument similar to Step 1 yields $w_h \leq
                         v_h$ in $\omega_{h,\delta}$. Moreover,
                         recalling \Cref{L:FullConsistency}
                         (consistency of $T_\varepsilon[\interp u]$) we have for all $x_i\in\Nhi$
\[
T_\varepsilon[v_h](x_i) = f(x_i) \le
T_\varepsilon[\interp u](x_i) + e_i,
\quad
T_\varepsilon[w_h](x_i) \ge T_\varepsilon[\interp u](x_i) + \sigma^d,
\]
where $e_i$ is a local bound for the consistency error.
Combining this with \Cref{P:ContDep} (continuous dependence on data) in $\Omega_{h,\delta}$ gives
\[
\max_{\Omega_{h,\delta}} d_\varepsilon^- \lesssim \delta
\left( \sum_{x_i\in\cC_-^\delta(d_\varepsilon^-)} \Big(
\big(T_\varepsilon[\interp u](x_i) + e_i\big)^{\frac{1}{d}}
- \big( T_\varepsilon[\interp u](x_i) + \sigma^d \big)^{\frac{1}{d}}
\Big)^d\right)^{\frac{1}{d}} ,
\]
Since $\interp u$ is discretely convex, we apply \Cref{L:DisConv}
(discrete convexity) to deduce $T_\varepsilon[\interp u](x_i)\ge 0$
and next argue as in Step 1 to obtain

$$
\interp u - \uve \lesssim \sigma + |u|_{\Wtio} \delta + \sigma^{1-d} \ \frac{\delta}{h} \left(  C_1(u) \delta^{s-2} + C_2(u) \  \left( \delta^{1/d} +  \frac{h^2}{\delta^2} + \theta^2  \right)  \right).
$$		 
			 \medskip
			 \textit{Step 3: Choice of $\delta, \theta$ and $\sigma$.}
                         Since $\|u-\interp u\|_{L^{\infty}(\Omega_h)}\leq
C|u|_{\Wtio} h^2$, combining Steps 1 and 2 yields
\begin{align*}
  \|\uve - u\|_{L^\infty(\Omega_h)} &\lesssim
  \sigma + |u|_{\Wtio} (\delta + h^2) \\
  & + \sigma^{1-d} \ \frac{\delta}{h} \left(  C_1(u) \delta^{s-2} + C_2(u) \  \left( \delta^{1/d} +  \frac{h^2}{\delta^2} + \theta^2  \right)  \right) .
\end{align*}
We now minimize the right-hand side upon choosing
$\delta, \theta$ and $\sigma$ suitably with respect to $h$.
We first recall the definition of $\beta$ and
choose $\delta$ and $\theta$ as in \Cref{T:RatesPW}.
At this stage it only remains to find $\sigma$ upon solving
$$
\sigma = C_2(u) \sigma^{1-d} \frac{h}{\delta} = C \sigma^{1-d}
|u|_{\Wtio}^d R(u)^{-1} \ h^{1-\frac{2}{\beta}},
$$
which leads to 
$$
 \sigma = |u|_{\Wtio} R(u)^{-\frac{1}{d}} \ h^{\frac{1}{d}\left(1-\frac{2}{\beta}\right)}.
$$
Since $\beta>2$ we get $h^2 + \delta \le \big(1 +
R(u)^{-\frac{1}{\beta}}\big) h^{\frac{2}{\beta}}$ and
\[
\|\uve - u\|_{L^\infty(\Omega_h)} \lesssim 
|u|_{\Wtio} \big(1 + R(u) \big) h^{\frac{2}{\beta}}
+ |u|_{\Wtio} R(u)^{-\frac{1}{d}} \ h^{\frac{1}{d}\left(1-\frac{2}{\beta}\right)}.
\]
This yields the asserted estimate and finishes the proof.
\end{proof}

\Cref{T:RatesDegen} is an extension of \Cref{T:RatesPW} to the
degenerate case $f\ge0$, but the same techniques and estimates
extend as well to Theorems \ref{T:RatesHolder} and \ref{T:RatesSobolev}. 
We stress that Theorems \ref{T:RatesPW} and \ref{T:RatesDegen} correspond to non-classical viscosity solutions that are of class $W^{2}_{\infty}(\Omega)$. In order to deal with discontinuous Hessians and degenerate right hand sides, we rely on techniques that give rise to reduced rates. For \Cref{T:RatesPW} we obtain rates that depend on the space dimension, whereas for \Cref{T:RatesDegen} we resort to a regularization procedure that leads to further reduction of the rates. Although the derived estimates are suboptimal with respect to the computational rates observed in \cite{NoNtZh}, we wish to emphasize that \Cref{T:RatesDegen} is, to our knowledge, the only error estimate available in the literature that deals with degenerate right hand sides.

\section{Conclusions}
	In this paper we extend the analysis of the two-scale method
        introduced in \cite{NoNtZh}. We derive
	continuous dependence of discrete solutions on data and use it to prove rates of 
	convergence in the $L^{\infty}$ norm in the computational domain $\Omega_h$ 
	for four different
        cases. We first prove rates of order up to $h^{1/2}$ for
        smooth classical solutions with H\"older regularity. We then
        exploit the structure of the continuous dependence
        estimate of discrete solutions on data to
        derive error estimates for classical solutions with Sobolev regularity, 
	thereby achieving the same rates under weaker regularity
        assumptions. In a more general scenario, we derive error estimates
	for viscosity solutions with discontinuous Hessian across a
        surface with appropriate smoothness, but otherwise possessing
        piecewise Sobolev regularity.
        Lastly, we use an
	interior barrier function that allows us to remove the
        nondegeneracy assumption
	$f > 0$ at the cost of a reduced rate that depends
        on dimension. Our theoretical
	predictions are sub-optimal with respect to the linear rates
        observed experimentally in 
	\cite{NoNtZh} for a smooth classical solution and a
        piecewise smooth viscosity solution with 
	degenerate right-hand side $f\ge 0$. This can be
        attributed to the fact that the continuous dependence estimate
        of discrete solutions on data introduces a factor $\frac{\delta}{h} \gg 1$
        in the error estimates. This feature is similar to the
        discrete ABP estimate developed in \cite{KuoTru} and is the
        result of using sets of measure $\approx \delta^d$ instead of
        $\approx h^d$ to approximate subdifferentials. In a
        forthcoming paper we will tackle this issue and connect our
        two-scale method with that of Feng and Jensen \cite{FeJe:16}.

	\bibliographystyle{amsplain}

\end{document}